\documentclass[12pt]{amsart}

\usepackage{etex}
\usepackage[usenames,dvipsnames]{pstricks} 
\usepackage{epsfig}
\usepackage{graphicx,color}
\usepackage{geometry}
\usepackage[all]{xy}
\usepackage{amssymb,amscd}
\usepackage{cite}
\usepackage{fullpage}
\usepackage{marvosym} 
\xyoption{poly}
\usepackage{url}
\usepackage{comment}
\usepackage{float}
\usepackage{bm}

\usepackage{tikz}
\usepackage{tikz-cd}
\usetikzlibrary{decorations.pathmorphing}

\newtheorem{theorem}{Theorem}[section]
\newtheorem{lemma}[theorem]{Lemma}
\newtheorem{proposition}[theorem]{Proposition}
\newtheorem{corollary}[theorem]{Corollary}
\theoremstyle{definition}
\newtheorem{definition}[theorem]{Definition}

\newtheorem{remark}[theorem]{Remark}

\newtheorem{example}[theorem]{Example}

\newtheorem*{question*}{Question}

\newtheorem*{questions*}{Questions}

\newtheorem*{steps*}{Answer/steps}

\newtheorem*{progress*}{Progress}

\newtheorem*{classification*}{Classification}

\newtheorem*{construction*}{Classification}
\newtheorem*{example*}{Example}

\newtheorem*{remark*}{Remark}
\newtheorem*{remarks*}{Remarks}
\newtheorem*{definition*}{Definition}

\usepackage{calrsfs}
\usepackage{url}
\usepackage{longtable}
\usepackage[OT2, T1]{fontenc}
\usepackage{textcomp}
\usepackage{times}
\usepackage[scaled=0.92]{helvet}

\newcommand{\C}{\mathbb{C}}
\newcommand{\Q}{\mathbb{Q}}

\newcommand{\R}{\mathbb{R}}
\newcommand{\Z}{\mathbb{Z}}

\newcommand{\X}{\mathcal{X}}

\DeclareMathOperator{\Res}{Res}

\DeclareMathOperator{\GL}{GL}

\DeclareMathOperator{\Aut}{Aut}

\DeclareSymbolFont{cyrletters}{OT2}{wncyr}{m}{n}
\DeclareMathSymbol{\Sha}{\mathalpha}{cyrletters}{"58}

\makeatletter

\def\greekbolds#1{%
 \@for\next:=#1\do{%
    \def\X##1;{%
     \expandafter\def\csname V##1\endcsname{\boldsymbol{\csname##1\endcsname}}
     }
   \expandafter\X\next;
  }
}

\greekbolds{alpha,beta,iota,gamma,lambda,nu,eta,Gamma,varsigma,Lambda}

\def\make@bb#1{\expandafter\def
  \csname bb#1\endcsname{{\mathbb{#1}}}\ignorespaces}

\def\make@bbm#1{\expandafter\def
  \csname bb#1\endcsname{{\mathbbm{#1}}}\ignorespaces}

\def\make@bf#1{\expandafter\def\csname bf#1\endcsname{{\bf
      #1}}\ignorespaces} 

\def\make@gr#1{\expandafter\def
  \csname gr#1\endcsname{{\mathfrak{#1}}}\ignorespaces}

\def\make@scr#1{\expandafter\def
  \csname scr#1\endcsname{{\mathscr{#1}}}\ignorespaces}

\def\make@cal#1{\expandafter\def\csname cal#1\endcsname{{\mathcal
      #1}}\ignorespaces} 

\def\do@Letters#1{#1A #1B #1C #1D #1E #1F #1G #1H #1I #1J #1K #1L #1M
                 #1N #1O #1P #1Q #1R #1S #1T #1U #1V #1W #1X #1Y #1Z}
\def\do@letters#1{#1a #1b #1c #1d #1e #1f #1g #1h #1i #1j #1k #1l #1m
                 #1n #1o #1p #1q #1r #1s #1t #1u #1v #1w #1x #1y #1z}
\do@Letters\make@bb   \do@letters\make@bbm
\do@Letters\make@cal  
\do@Letters\make@scr 
\do@Letters\make@bf \do@letters\make@bf   
\do@Letters\make@gr   \do@letters\make@gr
\makeatother

\def\ol{\overline}

\newcommand{\isoto}{\stackrel{\sim}{\longrightarrow}}

\def\Qbar{\overline{\bbQ}}

\def\makeop#1{\expandafter\def\csname#1\endcsname
  {\mathop{\rm #1}\nolimits}\ignorespaces}
\makeop{Hom}   \makeop{End}   \makeop{Aut}   \makeop{Isom}  \makeop{Pic} 
\makeop{Gal}   \makeop{ord}   \makeop{Char}  \makeop{Div}   \makeop{Lie} 
\makeop{PGL}   \makeop{Corr}  \makeop{PSL}   \makeop{sgn}   \makeop{Spf}
\makeop{Spec}  \makeop{Tr}    \makeop{Nr}    \makeop{Fr}    \makeop{disc}
\makeop{Proj}  \makeop{supp}  \makeop{ker}   \makeop{im}    \makeop{dom}
\makeop{coker} \makeop{Stab}  \makeop{SO}    \makeop{SL}    \makeop{SL}
\makeop{Cl}    \makeop{cond}  \makeop{Br}    \makeop{inv}   \makeop{rank}
\makeop{id}    \makeop{Fil}   \makeop{Frac}  \makeop{GL}    \makeop{SU}
\makeop{Nrd}   \makeop{Sp}    \makeop{Tr}    \makeop{Trd}   \makeop{diag}
\makeop{Res}   \makeop{ind}   \makeop{depth} \makeop{Tr}    \makeop{st}
\makeop{Ad}    \makeop{Int}   \makeop{tr}    \makeop{Sym}   \makeop{can}
\makeop{length}\makeop{SO}    \makeop{torsion} \makeop{GSp} \makeop{Ker}
\makeop{Adm}   \makeop{Mat}

\DeclareMathSymbol{\twoheadrightarrow} {\mathrel}{AMSa}{"10}

\begin{document}

\title{Uniqueness of indecomposable idempotents in algebras with involution}

\author{Valentijn Karemaker}
\address{Mathematical Institute, Utrecht University, Utrecht, The Netherlands}
\email{V.Z.Karemaker@uu.nl}

\author{Akio Tamagawa}
\address{Research Institute for Mathematical Sciences, Kyoto University, Kyoto, Japan}
\email{tamagawa@kurims.kyoto-u.ac.jp}

\author{Chia-Fu Yu}
\address{Institute of Mathematics, Academia  Sinica and National Center for Theoretic Sciences, Taipei, Taiwan}
\email{chiafu@math.sinica.edu.tw}


\maketitle
\setcounter{tocdepth}{2}

\begin{abstract}
We prove uniqueness of a decomposition of $1$ into indecomposable Hermitian idempotents in an order of a finite-dimensional $\mathbb{Q}$-algebra with positive involution, by generalising a result of Eichler on unique decomposition of lattices. We use this result to prove that polarised abelian varieties over any field admit a unique decomposition into indecomposable polarised abelian subvarieties, a result previously shown by Debarre and Serre with different methods and over algebraically closed fields. We prove that an analogous uniqueness result holds true for arbitrary polarised integral Hodge structures, and derive a consequence for their automorphism groups.
\end{abstract}


\section{Introduction}

Let $R^0$ be a finite-dimensional $\Q$-algebra with anti-involution $*$, that is, $(xy)^*=y^* x^*$ for $x,y \in R^0$. 
Let $R\subseteq R^0$ be an order of $R^0$, that is, $R$ is a subring of $R^0$ which is a $\Z$-lattice (of full rank). Let $\Tr_{R^0/\Q}:R^0\to \Q$ denote the usual (left) trace form.
Denote by
\[ I(R):=\{i\in R: i^2=i \}\supseteq I^*(R):=\{i\in I(R): i^*=i\} \]
the sets of idempotents in $R$ and Hermitian idempotents in $R$, respectively. We say $x,y \in R$ are \emph{orthogonal}, and write $x\bot y$, if $xy^*=0$. We call a nonzero element $i \in I^*(R)$ \emph{indecomposable} if whenever $i=j+k$ with $j,k\in I^*(R)$ such that $j\bot k$, either $j=0$ or $k=0$. In this note we prove the following result on the uniqueness of Hermitian idempotents in $R$.

\begin{theorem}\label{thm:UniDecIdem}
Assume that $\Tr_{R^0/\Q}(x x^*)>0$ for any $x\neq 0\in R^0$.  
If 
\[ 1=i_1+\dots +i_r, \quad \text{ for } i_\nu \in I^*(R) \text{ indecomposable and } i_\nu\bot i_{\nu'} \text{ whenever } \nu\neq \nu', \]  
then $\{i_\nu\}_{\nu}$ is uniquely determined. 
\end{theorem}

\begin{remark}
    Under the assumption of Theorem~\ref{thm:UniDecIdem}, the left trace form $\Tr_{R^0/\Q}$ coincides with the right trace form $\Tr_{(R^0) ^\mathrm{opp}}/\Q$; 
    see Proposition~\ref{prop:relation}.
\end{remark}

Let $V$ be a finite-dimensional $\Q$-vector space with a positive-definite symmetric form $f:V\times V \to \Q$. A classical result of Eichler \cite[Satz 2]{eichler:1952} (see also Kneser~\cite {Kneser:1954},\cite[Satz 27.2]{kneser}, and~Kitaoka~\cite[Theorem 6.7.1]{kitaoka})
shows uniqueness of an orthogonal decomposition of any lattice in $V$:


\begin{theorem}\label{thm:1.2}
Let $L$ be a $\Z$-lattice in $(V,f)$. If
\begin{equation}\label{eq:unique_dec_L}
L=\bot_{i=1}^r L_i     
\end{equation}
is an orthogonal decomposition into indecomposable $\Z$-sublattices in $V$, then the set of sublattices $L_i$ is uniquely determined by $L$, not just determined up to isomorphism.
\end{theorem}

\noindent Following Theorem~\ref{thm:1.2}, if one writes \eqref{eq:unique_dec_L} as 
\begin{equation}\label{eq:uniqu_dec_L:2}
(L,f) \simeq \prod_{i=1}^r (L_i, f_i)^{e_i} \ \  
\end{equation}
where $(L_i, f_i)$ is an indecomposable $\Z$-sublattice with the restriction form $f_i$ of $f$, satisfying $(L_i,f_i)\not\simeq (L_j,f_j)$ for $i\neq j$,  
then we have 
\begin{equation}
    \label{eq:auto_grp_L}
    \Aut(L,f)\simeq \prod_{i=1}^r \Aut(L_i,f_i)^{e_i}\cdot S_{e_i}, 
\end{equation}
where $S_{e_i}$ denotes the symmetric group of the set $\{1,\dots, e_i\}$.\\

Inspired by the fact that some abelian varieties can be described by lattices, 
it is interesting to know whether  polarised abelian varieties also 
share analogous properties of $\Z$-lattices in a positive-definite symmetric space, namely uniqueness of an orthogonal decomposition as above. This is indeed the case, as previously proved by Debarre (using algebraic geometry) and by Serre (using algebraic methods) \cite{Debarre} over an algebraically closed field. 
As an application of 
Theorem~\ref{thm:UniDecIdem}, we also obtain the same result with an arbitrary ground field.

\begin{theorem}\label{thm:1.3}
    Every polarised abelian variety over any field $K$ admits the following unique decomposition property: if $(X,\lambda)=\prod_{i=1}^r (X_i,\lambda_i)$ is a decomposition into indecomposable polarised abelian subvarieties with their respective induced polarisations over $K$, then the set $\{X_i\}_i$ of abelian subvarieties is uniquely determined, not just determined up to isomorphism.
\end{theorem}

Similarly, this has the following consequence.

\begin{corollary}
If one writes $(X,\lambda)\simeq \prod_{i=1}^r (X_i,\lambda_i)^{e_i}$ as a product of indecomposable polarised abelian varieties over $K$ with $(X_i,\lambda_i)\not\simeq (X_j,\lambda_j)$ for $i\neq j$, then 
\begin{equation}\label{eq:autgp_pav}
    \Aut(X,\lambda)\simeq \prod_{i=1}^r \Aut(X_i,\lambda_i)^{e_i}\cdot S_{e_i}.
\end{equation}    
\end{corollary}

As a generalisation of the statement of Theorem~\ref{thm:1.3} over the ground field $\C$, we show the following result.

\begin{theorem}\label{thm:1.5}
Let $(L,\psi,h)$ be a polarised integral Hodge structure. If 
\begin{equation}\label{eq:PHdg_Z}
L=\bot_{i=1}^r L_i     
\end{equation}
is an orthogonal decomposition into indecomposable polarised  integral Hodge substructures, then the set of sublattices $L_i$ is uniquely determined by $L$.
\end{theorem}

We recall the definitions of polarised integral Hodge structures and their orthogonal decompositions in Section~\ref{sec:P}.

\begin{corollary}
If 
\begin{equation}\label{eq:dec_phs}
(L,\psi,h)\simeq \prod_{i=1}^r (L_i,\psi_i, h_i)^{e_i}     
\end{equation}
is an orthogonal decomposition into indecomposable polarised integral Hodge substructures with $(L_i,\psi_i, h_i)\not\simeq (L_j,\psi_j, h_j)$ for $i\neq j$, 
then 
\begin{equation}\label{eq:autgp_phs}
    \Aut(L,\psi, h)\simeq \prod_{i=1}^r \Aut(L_i,\psi_i,h_i)^{e_i}\cdot S_{e_i}.
\end{equation}   
\end{corollary}

In Section 2, we organise the relations among semi-simplicity of a finite-dimensional algebra (with or without involution), and non-degeneracy and positive-definiteness of the trace forms in question. 
To prove Theorem~\ref{thm:UniDecIdem}, in Section~3 we generalise Eichler's result (Theorem~\ref{thm:1.2}) to positive-definite Hermitian $B$-modules for an arbitrary semi-simple $\Q$-algebra with positive involution (Theorem~\ref{orthogonal}). In Sections 4, 5 and~6, respectively, we show that Theorem~\ref{orthogonal} implies Theorem~\ref{thm:UniDecIdem}, that Theorem~\ref{thm:UniDecIdem} implies Theorem~\ref{thm:1.3}, and that  Theorem~\ref{thm:UniDecIdem} implies Theorem~\ref{thm:1.5}.

\subsection*{Acknowledgements}
The first author was partially supported by the Dutch Research Council (NWO) through grant VI.Veni.192.038.
The second author was supported by JSPS KAKENHI Grant Number 20H01796. 
The third author was supported by
the NSTC grants 109-2115-M-001-002-MY3. 
This work was supported by the Research Institute for Mathematical Sciences, an International Joint Usage/Research Center located in Kyoto University.

\section{Traces, involutions, semi-simplicity, non-degeneracy and positive-definiteness}
\subsection{Various traces of finite-dimensional algebras}\

Let $F$ be a field of characteristic $0$ and~$A$ a finite-dimensional $F$-algebra.
Let $L$ be a field extension of $F$ and
set $A_L:=A\otimes_F L$, a finite-dimensional $L$-algebra.
Let~$V$ be a finite-dimensional $L$-vector space of dimension $n$.
We let $\End_L(V)$ act on $V$ from the left. 
Let $\Tr_V: \End_L(V) \to L$ denote  the usual trace, via $\End_L(V)\simeq \Mat_n(L)$. 

For any representation $\rho: A_L \to \End_L(V)$, i.e.~a homomorphism of $L$-algebras, we denote
$\Tr_{\rho}: =\Tr_V\circ \rho: A_L \to L$.
By abuse of notation, we denote $\Tr_{\rho}|_A: A\to L$ also by $\Tr_{\rho}$.

\begin{example}\
\begin{enumerate}
\item Let $L=F, V=A$ and $\rho_A^{\rm left}: A\to \End_F(A), a \mapsto l_a$, where $l_a(x)=ax$ for $x\in A$ is the left $a$-multiplication.
We call 
\[ \Tr_{A/F}:=\Tr_{\rho_A^{\rm left}}: A\to F \] 
the usual left trace.
This is intrinsically determined only by $A/F$.

\item We call $\Tr_{A^{\rm opp}/F}: A=A^{\rm opp}\to F$ the usual right trace. The trace $\Tr_{A^{\rm opp}/F}(a)$
coincides with $\Tr_A(r_a)$, where $r_a$ is the right $a$-multiplication.
This is intrinsically determined only by $A/F$.

\begin{remark}\
    \begin{itemize}
        \item[(a)] Let $A$ be the 3-dimensional $F$-subalgebra of $\Mat_2(F)$ consisting of all upper-triangular matrices. Then, for $a = (a_{ij})\in A$ (so that $a_{21}=0), \Tr_{A/F}(a)=2a_{11}+a_{22}$, while $\Tr_{A^{\rm opp}/F}(a)=a_{11}+2a_{22}$.  
        \item[(b)] If $A$ is semi-simple, we have $\Tr_{A/F}=\Tr_{A^{\rm opp}/F}$; see Proposition~\ref{prop:relation}.
    \end{itemize}
\end{remark}

\item Assume $A$ is a central simple algebra over $F$. Take an isomorphism $\rho: A_{\overline F} \to \Mat_n(\overline F)$ of $\overline F$-algebras. It is well known that $\Trd_{A/F}:=\Tr_{\rho}: A \to \overline F$ does not depend on the choice of $\rho$, and that its image is contained in $F$, hence it induces $\Trd_{A/F}: A \to F$, which we call the reduced trace.
This is intrinsically determined only by $A/F$ and one has $\Tr_{A/F}=\Tr_{A^{\rm opp}/F}=n \Trd_{A/F}$.

\item More generally, assume $A$ is semi-simple. Then we can take an isomorphism $\rho: A_{\overline F} \to \prod_{i=1}^r \Mat_{n_i}(\overline F)$. By abuse of notation, we denote the composite of $\rho$ and the (block-diagonal) embedding $\prod_{i=1}^r \Mat_{n_i}(\overline F)\to \Mat_{n_1+...+n_r}(\overline F)$ again by $\rho$. One can prove that $\Trd_{A/F}:=\Tr_{\rho}: A \to \overline F$ does not depend on the choice of $\rho$ and that its image is contained in $F$, hence it induces $\Trd_{A/F}: A \to F$. We also call $\Trd_{A/F}: A \to F$ the reduced trace.
This is intrinsically determined only by $A/F$.

\item Let $F=\mathbb{Q}$. Let $X$ be an abelian variety over a field $K$, and set $A:=\End_K(X)_{\mathbb Q}$ to be the endomorphism algebra of $X$ over $K$. Put $L=\mathbb{Q}_l$ and  $V=V_l(X)$. Then we have a natural representation $\rho_l: A_L\to \End_L(V)$, which induces $\Tr_{\rho_l}: A\to \Q_l$. It is well known that the image of $\Tr_{\rho_l}$ is contained in $\mathbb{Q}$, hence it induces $\Tr_{\rho_l}: A\to\mathbb Q$, and that $\Tr_{\rho_l}: A\to\mathbb Q$ is independent of the choice of $l$. We write $\Tr_X:=\Tr_{\rho_l}: A\to \mathbb{Q}$.

We remark that the map $\Tr_X$ is \emph{not} intrinsically determined only by $A/\mathbb{Q}$. For example, let $X$ be a $g$-dimensional abelian variety over $K$ with $\End_K(X)=\mathbb{Z}$, hence $A:=\End_K(X)_{\mathbb Q}=\mathbb Q$. Then $\Tr_X: \mathbb{Q} \to \mathbb{Q}$ turns out to be the $2g$-multiplication, which is not determined only by $A/\mathbb{Q}$.

For example, let $Y$ be a QM abelian surface over K with $\End_K(Y)$ an order of a quaternion algebra $D$ over $\mathbb{Q}$, and let $Z$ be an $h$-dimensional abelian variety over $K$ with $\End_K(Z)=\mathbb{Z}$. Consider the $(4+3h)$-dimensional abelian variety $X:=Y^2\times Z^3$ over~$K$. Then $A:=\End_K(X)_{\mathbb Q}$ is isomorphic to $\Mat_2(D)\times \Mat_3(\mathbb{Q})$. For simplicity, let $p: A\to \Mat_2(D)$ and $q: A\to \Mat_3(\mathbb{Q})$ be projections, and set $\alpha:=\Trd_{\Mat_2(D)/\mathbb{Q}} \circ p$ and $\beta:=\Trd_{\Mat_3(\mathbb{Q})/\mathbb{Q}}\circ q$. (Thus, $\alpha, \beta: A\to\mathbb{Q}$.) Then one calculates:
\[ \Tr_{A/\mathbb Q}=\Tr_{A^{\rm opp}/\mathbb{Q}}=4\alpha+3\beta, \quad
\Trd_{A/\mathbb{Q}}=\alpha+\beta, \quad
\Tr_X=2\alpha+(2h)\beta. \]
\end{enumerate}
\end{example}
\subsection{Relations between semi-simplicity, non-degeneracy and positive-definiteness}\

Let $F$ be a field of characteristic $0$ and $A$ be a finite-dimensional $F$-algebra.

\begin{proposition}\label{prop:relation} \ 
\begin{enumerate}
    \item Consider the following conditions:
\begin{itemize}
    \item[(nd)] The map $A\times A \to F, (x,y)\mapsto \Tr_{A/F}(xy)$ is non-degenerate.
    \item[(ss)] $A$ is semi-simple.
     \item[(l=r)] $\Tr_{A/F}=\Tr_{A^{\rm opp}/F}$.  
\end{itemize}
Then {\rm (nd)} $\iff$ {\rm (ss)} $\implies$ {\rm (l=r)}.

\item Assume that an involution $*: A \to A$ (with $(xy)^*=y^*x^*$) is given. Consider the following conditions:

\begin{itemize}
    \item[(nd*)] The map $A\times A \to F, (x,y)\mapsto \Tr_{A/F}(xy*)$ is non-degenerate.
    \item[(l=l*)] $\Tr_{A/F}(a)=\Tr_{A/F}(a^*)$ for any $a\in A$.
\end{itemize}
Then {\rm (nd*)} $\iff$ {\rm (nd)}; and  {\rm (l=l*)} $\iff$ {\rm (l=r)}.

\item Assume that an involution $*: A\to A$ (with $(xy)^*=y^*x^*$) is given and that $F$ is a subfield of $\mathbb{R}$. Consider the following condition:

\begin{itemize}
    \item[(pd*)]  $\Tr_{A/F}(xx^*)>0$ for any $x \in A -\{0\}$.
\end{itemize}
Then {\rm (pd*)} $\implies$ {\rm (nd*)}.    
\end{enumerate}
\end{proposition}

\begin{proof}
\begin{enumerate}
    \item Let $J$ denote the Jacobson radical of $A$, which is defined to be the intersection of all maximal left ideals of $A$ and is known to coincide with the intersection of all maximal right ideals of A. (In particular, $J$ is a two-sided ideal of $A$.)  The following facts are well known:
    \begin{itemize}
        \item[(i)] $J$ is a nilpotent ideal, i.e.~$J^n=0$ for some $n$. In particular, any element of $J$ is nilpotent. \item[(ii)] $A/J$ is semi-simple. 
        \item[(iii)] $A$ is semi-simple if and only if $J=0$.
    \end{itemize}
Since $x\in J$ is nilpotent, the left $x$-multiplication $ l_x \in \End_F(A)$ is also nilpotent, hence $\Tr_{A/F}(x)=0$. Thus, $\Tr_{A/F}(J)=0$. Since $AJ=J$, the map $A\times J \to F, (x,y)\mapsto \Tr_{A/F}(xy)$ is zero. In particular, (nd) implies $J=0$, which by (iii) implies (ss).

Next, assume (ss) holds. To show (nd) and (l=r) hold for $A/F$, it suffices to show (nd) and (l=r) hold for $A_{\overline F}/\overline{F}$. But since (ss) implies that $A_{\overline F}$ is isomorphic to $\prod_{i=1}^r \Mat_{n_i}(\overline F)$, we see that it suffices to show (nd) and (l=r) for $\Mat_n(\overline F)$, which are elementary.

\item Since $*: A\to A$ is an isomorphism of $F$-vector spaces, (nd*) $\iff$ (nd) is clear. Next, since the involution $*: A\to A$ yields an isomorphism of $F$-algebras $A \to A^{\rm opp}$, we see that $\Tr_{A/F}\circ * =\Tr_{A^{\rm opp}/F}$. Thus, (l=l*) $\iff$ (l=r) follows.

\item Assume (pd*) holds. Let $x\in A$. If $\Tr_{A/F}(xy^*)=0$ for all $y\in A$, then in particular $\Tr_{A/F}(xx^*)=0$, hence $x=0$ by (pd*). Similarly, if $\Tr_{A/F}(yx^*)=0$ for all $y\in A$, then in particular $\Tr_{A/F}(xx^*)=0$, hence $x=0$ by (pd*).
\end{enumerate}
\end{proof}

\begin{remark}
The condition (l=r) does not imply (ss) in general. Any commutative finite-dimensional $F$-algebra admitting nontrivial nilpotent elements provides an example of this.    
\end{remark}

\section{Uniqueness of orthogonal decompositions of positive-definite Hermitian lattices}

Let $B$ be a finite-dimensional 
$\Q$-algebra together with a positive involution $*$, that is, such that 
$\Tr_{B/\Q} (b b^*)>0$ for any nonzero element $b\in B$. By Proposition~\ref{prop:relation}, the algebra $B$ is automatically semi-simple. Note that when $B$ is semi-simple, positivity of $*$ is equivalent to the condition that $\Trd_{B/\Q} (b b^*)>0$ for any nonzero element $b\in B$, where $\Trd_{B/\Q}$ denotes the reduced trace from $B$ to $\Q$. However, we shall not use the reduced trace $\Trd_{B/\Q}$ in this paper.
We fix an order $O$ of $B$, that is, $O$ is a subring of $B$ that is finitely generated as a $\Z$-module and spans $B$ over $\Q$.   



Let $V$ be a faithful finite left $B$-module. In particular, $V\neq 0$. 
Let $f:V\times V\to B $ be a ($B$-valued) Hermitian form on $V$. This means that $f$ is a $\Q$-bilinear pairing on $V$ satisfying 
\begin{equation}\label{eq:hermitian}
    \begin{split}
     f(ax,y) & =af(x,y) \quad \text{and} \quad f(y,x)=f(x,y)^* \quad  \text{ for any $x$, $y\in V$, $a\in B$}.  
    \end{split}
\end{equation}
We call the Hermitian form $f$  \emph{positive definite} if the composition $\Tr_{B/\Q} \circ f:V\times V\to \Q$ is positive definite. 

We now assume that the Hermitian form $f$ is positive definite. 
A $\Z$-lattice is a finitely generated torsion-free $\Z$-module; a $\Z$-lattice in a finite-dimensional $\Q$-vector space $W$ means a finitely generated $\Z$-submodule in $W$ which spans $W$ over $\Q$. Further, a (left) $O$-lattice is a finitely generated left $O$-module which is a $\Z$-lattice, and an $O$-lattice $L$ in $V$ is a $\Z$-lattice in~$V$ which is also an $O$-submodule of $V$, that is, $O L \subseteq L$. 
An $O$-submodule~$M$ of an $O$-lattice $L$ is called an $O$-sublattice of $L$; then $M$ is an $O$-lattice in the $B$-submodule $B M$ of $V$, possibly of smaller dimension over $\Q$.

\begin{definition}\label{def:3.1}
\begin{enumerate}
    \item A nonzero $O$-lattice $L$ in $V$ is \emph{indecomposable} if whenever
$L=L_1+L_2$ and $f(L_1,L_2)=0$ for some $O$-sublattices $L_1$ and $L_2$, one has either $L_1=0$ or $L_2=0$. 
    \item Let $L$ be an $O$-lattice in $V$. A nonzero element $x\in L$ is \emph{primitive} if whenever $x=y+z$ for some elements $y,z\in L$ with $f(y,z)=0$, one has either $y=0$ or $z=0$.  
    \item Two primitive elements $x$ and $y$ in $L$ (which may coincide) are \emph{connected}, denoted $x\sim y$, if there exist primitive elements $z_0,\dots, z_n\in L$ such that $z_0=x$, $z_n=y$ and $f(z_{i-1},z_i)\neq 0$ for $i=1,\dots, n$. One sees that $\sim$ is an equivalence relation.  
\end{enumerate}
\end{definition}

\begin{remark}
    An alternative definition of primitive elements of $L$ is as follows: one replaces the condition $f(y,z)=0$ in Definition~\ref{def:3.1}(2) with $\Tr_{B/\Q}(f(Oy,Oz))=0$.
    To see the equivalence, first note that if $f(y,z)=0$, then $f(Oy,Oz)=0$ and $\Tr_{B/\Q}(f(Oy,Oz))=0$. Conversely, if $\Tr_{B/\Q}(f(Oy,Oz))=0$, then $\Tr_{B/\Q}(Of(y,z))=\Tr_{B/\Q}(f(Oy,z))=0$ and hence $f(y,z)~=~0$, by non-degeneracy of 
    $\Tr_{B/\Q}(xy)$.    
\end{remark}

\begin{lemma}\label{lm:prim}
Let $(V,f)$ be a positive-definite Hermitian $B$-module, and $L$ an $O$-lattice in $V$.
Then every nonzero element $x\in L$ is a finite sum of primitive elements in $L$. 
\end{lemma}
\begin{proof}
   If $x$ is primitive, then we are done.
   If $x$ is not primitive, then $x=y+z$ with nonzero elements $y$ and $z$ in $L$ such that $f(y,z)=0$. Then we have $f(x,x)=f(y,y)+f(z,z)$ and hence
\[ \Tr_{B/\Q} f(x,x)= \Tr_{B/\Q} f(y,y)+ \Tr_{B/\Q} f(z,z).\]
Since $f$ is positive definite, 
we have $\mathrm{Tr}_{B/\Q}(f(y,y))<\mathrm{Tr}_{B/\Q}(f(x,x))$. 
If $y$ or $z$ is not primitive, we continue with the same process. 
We claim  that this process terminates after finitely many steps.
Since $L\neq 0$ is a finitely generated $\Z$-submodule of~$V$ and $f$ is positive definite,  
$f(L,L)$ is a nonzero finitely generated $\Z$-submodule of~$B$, the module $\mathrm{Tr}_{B/\Q}(f(L,L)) := \{ \mathrm{Tr}_{B/\Q}(f(m,n)) : m,n \in L \}$ is 
a nonzero finitely generated $\mathbb{Z}$-submodule of~$\mathbb{Q}$,
and we have 
$\mathrm{Tr}_{B/\Q}(f(L,L))=e\Z$ for some $e\in \Q_{>0}$ . This means that $\mathrm{Tr}_{B/\Q}(f(x,x))\in e\Z_{\ge 0}$ for any $x\in L$.  
So after finitely many steps, the element becomes primitive 
and the claim is proved. 
\end{proof}


The following result generalises the original result of Eichler (Theorem~\ref{thm:1.2}) 
and its extensions~\cite[Theorem 2.4.9]{KirschmerHab} and \cite[Theorem 2.1]{ibukiyama-karemaker-yu} to the setting where $B$ is either a totally real field, 
a CM field, or a totally definite quaternion algebra over a totally real field with its canonical involution, see also \cite[Remark~2.3]{ibukiyama-karemaker-yu}. 
We adapt the ideas of the proofs of Eichler~\cite{eichler:1952} and Kneser~\cite{Kneser:1954}.

\begin{theorem}\label{orthogonal}
Assumptions and notation being as above, 
any $O$-lattice $L\subseteq V$ has an orthogonal decomposition
\[
L=L_1\perp \cdots \perp L_r
\]
into indecomposable $O$-sublattices.
The set of lattices $\{L_i\}_{1\leq i\leq r}$ is uniquely determined by~$L$.
\end{theorem}

\begin{proof}

Let $S$ be the set of all primitive elements of $L$, and let $\{S_\lambda\}_{\lambda \in \Lambda}$ be the set of equivalence classes of $S$ with respect to $\sim$. 
Let $L_\lambda$ be the $O$-sublattice generated by $S_\lambda$.
Note that $f(S_\lambda,S_{\mu})=0$ for $\lambda\neq \mu \in \Lambda$. Indeed, if $x,y$ are primitive and $f(x,y)\neq 0$, then $x$ and $y$ are connected.
Using Equation~\eqref{eq:hermitian}, we see that $f(L_\lambda,L_\mu)=0$. 

Since $S$ generates $L$ as a $\Z$-module by Lemma~\ref{lm:prim}, we obtain a decomposition into $O$-sublattices
\[ L=\bot_{\lambda\in \Lambda} L_\lambda. \]
Using dimension counting, we see that $\Lambda$ is a finite set. 
Next, we show that any primitive element in~$L_{\lambda}$ belongs 
to $S_{\lambda}$. If $y\in L_{\lambda}\subseteq L$ is primitive, 
then $y\in S_{\mu}$ for some $\mu\in \Lambda$, but if $\lambda\neq \mu$, then 
$y\in S_{\mu}\subseteq L_{\mu}$, so $f(y,y)=0$ and $y=0$, a contradiction.
Now if $L_{\lambda}=N_1\perp N_2$ for some $O$-sublattices $N_1\neq 0$, $N_2\neq 0$,
then whenever $x+y$ with $x\in N_1$, $y\in N_2$ is primitive, we have $x=0$ or $y=0$. This means that any element $x\in S_\lambda$ lies in either $N_1$ or $N_2$.
So if $x,y \in S_\lambda$, $f(x,y)\neq 0$ and $x\in N_1$, then $y \in N_1$; hence $S_{\lambda}\subseteq N_1$ and 
$N_1=L_{\lambda}$, a contradiction, so $L_{\lambda}$ is indecomposable.
Now if $L=\perp_{\kappa \in K}M_{\kappa}$ for other indecomposable 
lattices $M_{\kappa}$, then any primitive element $x$ of $L$ is contained 
in some $M_{\kappa}$ by definition of primitivity. 
By the same reasoning as before, if $x \in M_{\kappa}$ is primitive, then 
any primitive $y\in L$ connected to $x$ belongs to $M_{\kappa}$. 


This means that there is a map $\iota: \Lambda \to K$ such that $L_{\lambda}\subseteq M_{\iota(\lambda)}$. Since 
\[L = \perp_{\lambda\in\Lambda}L_{\lambda} \subseteq \perp_{\kappa\in \iota(\Lambda)}M_\kappa \subseteq  \perp_{\kappa\in K}M_{\kappa} = L, 
\]
we have $M_{\kappa}=\perp_{\lambda\in\iota^{-1}(\kappa)}L_{\lambda}$. 
Now, as $M_{\kappa}$ is indecomposable, $\iota$ must be a bijection and $L_{\lambda}=M_{\iota(\lambda)}$. 
\end{proof}

\begin{remark}
If we let $L^{\Z}_\lambda$ be the $\Z$-lattice generated by $S_\lambda$, then by the same reasoning, we have $L=\bot_{\lambda} L^\Z_\lambda$. Since $L^{\Z}_\lambda\subseteq L_\lambda$ and $\perp_{\lambda} L_\lambda^\Z=L=\perp_{\lambda} L_\lambda$, one obtains $L^{\Z}_\lambda=L_\lambda$. This shows that the $\Z$-lattice $L^{\Z}_\lambda$ is automatically an $O$-lattice. 
\end{remark}

\begin{corollary}\label{autodecomposition}
Assumptions and notation being as in Theorem~\ref{orthogonal}, suppose that
$L$ has an orthogonal decomposition 
\[
L=\perp_{i=1}^{r} \left (\perp_{j=1}^{e_i}L_{ij} \right )
\]
where the $L_{ij}$ are indecomposable $O$-sublattices of $L$, arranged in such a way that 
$L_{ij}$ is isometric to $ L_{i'j'}$ if and only if $i=i'$. 
Then we have 
\[
\Aut(L)\cong \prod_{i=1}^{r}\Aut(L_{i1})^{e_i}\cdot S_{e_i},
\]
where $S_{e_i}$ is the symmetric group on $e_i$ letters and 
$\Aut(L_{i1})^{e_i}\cdot S_{e_i}$ is a semi-direct product 
with respect to the natural permutation action of $S_{e_i}$ on $\Aut(L_{i1})^{e_i}$.
\end{corollary}

\begin{proof}
By Theorem \ref{orthogonal}, we see that for any element $\epsilon \in \Aut(L)$ and $i=1,\ldots, r$, 
there exists $\tau_i\in S_{e_i}$ such that 
$\epsilon(L_{ij})=L_{i\tau_i(j)}$ for $j=1, \ldots, e_i$, so the result follows.
\end{proof}

\section{Proof of Theorem~\ref{thm:UniDecIdem}}\label{sec:Pf11}

Let the notation be as in Introduction. Note that for $i,j\in I^*(R)$, the condition $ij=0$ is equivalent to either $ij^*=0$, $i^* j^*=0 $, or $ji=0$.

Since $\Tr_{R^0/\Q}(x x^*)>0$ for any nonzero $x\in R^0$, by Proposition~\ref{prop:relation}, $R^0$ is a semi-simple $\Q$-algebra and 
\begin{equation}
    \Tr_{R^0/\Q}(x^*)=\Tr_{R^0/\Q}(x) \quad \forall\, x\in R^0. 
\end{equation}

Put $V=R^0$ as a left $R^0$-module; then $L:=R$ is an $R$-lattice in $V$. Define the $\Q$-bilinear pairing $f:V\times V\to R^0$ by $f(x,y):=xy^*$. Since $*$ is a positive involution, 
$f$ is an $R^0$-valued positive-definite Hermitian form. 

Let $1=i_1+\dots +i_r$ be a decomposition with nonzero $i_\nu\in I^*(R)$ and $i_\nu \bot i_{\nu'}$ for $\nu\neq \nu'$. Put $L_\nu:=R i_\nu$. Since $f(R i_\nu, R i_{\nu'})=Ri_\nu i_{\nu'}^* R=0$ for $\nu\neq \nu'$, we obtain an orthogonal decomposition of $L$ into $R$-sublattices:
\begin{equation}\label{eq:OrthDecL}
    L=L_1 \bot \dots \bot L_r.
\end{equation}
Conversely, suppose we have an orthogonal decomposition of $L$ as in \eqref{eq:OrthDecL}.
Consider the composition
\[ \begin{CD}
   \varphi_\nu: R @>{p_{L_\nu}}>> L_\nu @>{i_{L_\nu}}>> R ,
   \end{CD}
\]
where ${p_{L_\nu}}$ is the projection onto $L_{\nu}$ and $i_{L_\nu}$ is the natural inclusion from $L_\nu$ to $R$.
Since $\varphi_\nu$ is $R$-linear, there is a unique element $i_\nu\in R$ such that $\varphi_\nu(x)=x \cdot i_\nu$, namely $i_\nu=\varphi_\nu(1)$. From $\varphi_\nu ^2=\varphi_\nu$, we obtain $i_\nu^2=i_\nu$.
From ${\rm id}_R=\varphi_1+\cdots+\varphi_r$, we obtain $1=i_1+\cdots+i_r$. We also have $L_\nu=\varphi_\nu(R)=R\cdot i_\nu$. And it follows from $L_\nu \bot L_{\nu'}$ for $\nu\neq \nu'$ that $i_{\nu} \cdot i_{\nu'}^*=f(i_\nu, i_{\nu'})=0$, so that $i_\nu \bot i_{\nu'}$. 

We deduce that $i_\nu^*=i_\nu$ for all $\nu$:
we have
$i_{\nu}=i_{\nu}\cdot 1^*=i_{\nu}(i_1^*+\cdots+i_r^*)=i_{\nu}i_{\nu}^*$, 
hence $i_{\nu}^*=(i_{\nu}i_{\nu}^*)^* =i_{\nu}i_{\nu}^*=i_\nu$. 

Thus, we have proved the following:
\begin{lemma}\label{lm:Dec_L_and_Dec_1}
    Let $(R,*)$ be as above.
    There is a one-to-one correspondence between orthogonal decompositions of $L=\bot_{\nu=1}^r L_\nu$ into $R$-sublattices and orthogonal decompositions $1=\sum_{\nu=1}^r i_\nu$ with $i_\nu\in I^*(R)$.
\end{lemma}

By Lemma~\ref{lm:Dec_L_and_Dec_1} and Theorem~\ref{orthogonal}, we have now proved Theorem~\ref{thm:UniDecIdem}. 

\section{Proof of Theorem~\ref{thm:1.3}}

Let $(X,\lambda)$ be a polarised abelian variety over any field $K$. Set $R:=\End_K(X)$ and $R^0:=R\otimes_\Z \Q$, and let $*$ be the Rosati involution induced by $\lambda$, i.e. $x^*:=\lambda^{-1}x^t\lambda$ for any $x \in R^0$. 
It is well known that the $\mathbb Q$-bilinear form $\Tr_X(xy^*)$ on $R^0$ (see Example 2.1.(5)) is positive definite, which implies that 
the $\Q$-bilinear form $\Tr_{R^0/\Q}(xy^*)$
is positive definite. 

Let $(X,\lambda)=\prod_{\nu=1}^r (X_\nu,\lambda_\nu)$ be a decomposition into nonzero polarised abelian subvarieties, 
namely, $X=\prod_{\nu=1}^r X_{\nu}$ is a decomposition into abelian subvarieties and $\lambda: X \to X^t$ decomposes as  $(\lambda_{\nu})_{\nu=1,...,r}: \prod_{\nu=1}^r X_{\nu} \to \prod_{\nu=1}^r X_{\nu}^t$.
For each $\nu$, put
\[ 
\begin{CD}
i_\nu: X @>{p_{X_\nu}}>> X_{\nu} @>{i_{X_v}}>> X, 
\end{CD}
\]
where $p_{X_\nu}$ is the projection of $X$ onto $X_{\nu}$ and $i_{X_\nu}:X_{\nu} \hookrightarrow X$ is the inclusion. By definition, we have $1=\sum_{\nu=1}^r i_{\nu}$. 

We show that $i_\nu\in I^*(R)$ for each~$\nu$. As $p_{X_{\nu}} \circ i_{X_{\nu}} = {\rm id}_{X_{\nu}}$, we have $i_{\nu}\in I(R)$. 
By the functorial properties of dual abelian varieties, we have
\[ i_{X_\nu}^t=p_{X_\nu^t}: X^t \to X_{\nu}^t, \quad p_{X_\nu}^t=i_{X_\nu^t}: X_{\nu}^t \to X^t. \]
As $\lambda=(\lambda_{\nu})_{\nu=1,...,r}$, we have $p_{X_{\nu}^t}\circ\lambda=\lambda_{\nu}\circ p_{X_{\nu}}$ and $\lambda\circ i_{X_{\nu}}=i_{X_{\nu}^t}\circ\lambda_{\nu}$, hence 
\[i_{\nu}^t\circ\lambda=(i_{X_{\nu}}\circ p_{X_{\nu}})^t\circ\lambda=i_{X_{\nu}^t}\circ p_{X_{\nu}^t}\circ\lambda=i_{X_{\nu}^t}\circ\lambda_{\nu}\circ p_{X_{\nu}}=\lambda\circ i_{X_{\nu}}\circ p_{X_{\nu}}=\lambda\circ i_{\nu}.
\]
Therefore, $i_{\nu}^*=\lambda^{-1} \circ i_{\nu}^t\circ\lambda=i_{\nu}$.  
Since $p_{X_{\nu}}\circ i_{X_{\nu'}}=0$ for $\nu\neq\nu'$, we also have
\[
i_{\nu}i_{\nu'}^*=i_{\nu}i_{\nu'}=i_{X_{\nu}}\circ p_{X_{\nu}}\circ i_{X_{\nu'}}\circ p_{X_{\nu'}}=0,
\]
so $i_{\nu}\perp i_{\nu'}$ for $\nu\neq\nu'$. 

Conversely, suppose that we have a decomposition $1=\sum_{\nu=1}^r i_\nu$ with nonzero $i_\nu\in I^*(R)$, such that $i_\nu\bot i_{\nu'}$ for $\nu\neq \nu'$. 
Write $X_\nu:=i_\nu(X)$; then $X=\prod_{\nu=1}^r X_\nu$ (hence $X^t=\prod_{\nu=1}^r X_{\nu}^t$). We claim this is a decomposition into polarised abelian varieties with respective polarisations $\lambda_\nu:=\lambda|_{X_\nu}$. 
More precisely, $\lambda_{\nu}=p_{X_{\nu}^t}\circ\lambda\circ i_{X_{\nu}}$, and we claim that $\lambda=(\lambda_{\nu})_{\nu=1,...,r}$, or, equivalently, 
\[
\lambda=\sum_{\nu=1}^r i_{X_{\nu}^t}\circ\lambda_{\nu}\circ p_{X_\nu}, 
\]
where $p_{X_\nu}, i_{X_\nu}, p_{X_{\nu}^t}, i_{X_{\nu}^t}$ are defined as above.

By hypothesis, we have 
\[
\lambda\circ i_{\nu}=\lambda\circ i_{\nu}^*=i_{\nu}^t\circ \lambda, 
\]
hence 
\[
\begin{split}
\lambda &=\lambda\circ\sum_{\nu=1}^r i_\nu = \lambda\circ \sum_{\nu=1}^r i_\nu^2 = \sum_{\nu=1}^r \lambda \circ i_\nu^2 \\
& =\sum_{\nu=1}^r i_{\nu}^t\circ\lambda\circ i_\nu=\sum_{\nu=1}^r (i_{X_\nu}\circ p_{X_{\nu}})^t\circ \lambda\circ(i_{X_\nu}\circ p_{X_{\nu}}) \\
& =\sum_{\nu=1}^r  i_{X_{\nu}^t} \circ p_{X_\nu^t} \circ \lambda\circ i_{X_\nu}\circ p_{X_{\nu}}=\sum_{\nu=1}^r i_{X_{\nu}^t}\circ \lambda_{\nu}\circ p_{X_{\nu}}, 
\end{split}
\]
as desired. 

One checks that one construction is the inverse of the other construction. This yields the following lemma.

\begin{lemma}\label{lm:corresp_dec} There is a one-to-one correspondence between the set of decompositions $(X,\lambda)=\prod_{\nu=1}^r (X_\nu,\lambda_\nu)$ into polarised abelian subvarieties and the set of decompositions 
$1=\sum_{\nu} i_\nu$ with nonzero $i_\nu\in I^*(R)$ and $i_\nu \bot i_{\nu'}$ for $\nu\neq \nu'$. Moreover, the polarised abelian subvariety 
$(X_\nu,\lambda_{\nu})$ is indecomposable if and only if the corresponding idempotent $i_v$ is indecomposable.      
\end{lemma}

By Theorem~\ref{thm:UniDecIdem} and Lemma~\ref{lm:corresp_dec}, 
Theorem~\ref{thm:1.3} is proved.

\section{Proof of Theorem~\ref{thm:1.5}} \label{sec:P}
In this section we prove uniqueness of an orthogonal decomposition of a polarised integral Hodge structure into a product of indecomposable polarised integral Hodge substructures.

\subsection{Polarised Hodge structures}\label{sec:P.1}

We recall the definition of polarised $\Z$-Hodge structures. Our references are Milne~\cite[Chapter~2]{milne:ShVar} and Deligne~\cite{deligne:shimura}. From now on, let $R=\Z, \Q$ or $\R$.

\begin{definition}
    Let $V$ be a real vector space of finite dimension and let $V_\C:=V\otimes \C$. 
    The complex conjugation of $\C$ acts on $V_\C$ via the factor $\C$, i.e., $\ol{v\otimes z}:=v\otimes \bar z$ for $v\in V$ and $z\in \C$.

\begin{enumerate}
    \item A {\it Hodge decomposition} of $V$ is a decomposition of the $\C$-vector space
    \[ V_\C=\bigoplus_{(p,q)\in \Z^2} V^{p,q}\] 
    into $\C$-subspaces $ V^{p,q}$ such that $\ol {V^{p,q}}=V^{q,p}$ for all $(p,q)$. 
    \item An {\it $\R$-Hodge structure} (or {\it real Hodge structure}) is a finite-dimensional real vector space $V$ together with a Hodge decomposition of $V$. The set $\{(p,q)\in \Z^2: V^{p,q}\neq 0\}$ is called the {\it Hodge type} of $V$. The $\R$-Hodge structure $(V, \{V^{p,q}\})$ is called {\it pure of weight $n$} if $V^{p,q}\neq 0$ only for $(p,q)$ with $p+q=n$. One has a unique decomposition
    \[ V=\bigoplus_n V^n, \quad \text{ where } \quad V^n_\C=\bigoplus_{p+q=n} V^{p,q}, \]
    and each $V^n$ is a pure $\R$-Hodge structure of weight $n$. 
    \item An {\it $R$-Hodge structure} is a finite and free $R$-module $L$ together with an $\R$-Hodge structure on $L_\R:=L\otimes_R \R$ such that
    $L=\bigoplus_{n\in \Z} (L\cap (L_\R)^n)$. 
    It is called \emph{pure of weight $n$} if the $\R$-Hodge structure $L_\R$ is so. Putting $L^n:=L\cap (L_\R)^n$, then $L^n$ is a pure $R$-Hodge structure of weight $n$.

    \item A \emph{morphism} $f:L_1 \to L_2$ of $R$-Hodge structures is an $R$-linear map such that the base-changed map $f_\C:L_{1,\C}\to L_{2,\C}$ preserves the $\Z^2$-graded structures of $L_{1,\C}$ and $L_{2,\C}$. 
\end{enumerate}
\end{definition}

Let $\bbS:=\Res_{\C/\R} \bbG_{m,\C}$ be the Deligne torus. One has $\bbS(\R)=\C^\times$ and $\bbS(\C) = (\C \otimes_{\mathbb{R}} \C)^\times \simeq \C^\times \times \C^\times$ (sending $a\otimes b\mapsto (ab,\bar a b)$) with the inclusion $\bbS(\R)\subset \bbS(\C)$ given by $z\mapsto (z, \bar z)$. Viewing $\bbS$ as an $\R$-torus, the complex conjugation acts on $\bbS(\C) = (\C \otimes_{\mathbb{R}} \C)^{\times}$ (through conjugating the second copy of~$\C$) 
by $\overline{(z_1,z_2)}=(\bar z_2, \bar z_1)$.
For each real Hodge structure $V$, one defines a morphism 
\[ h_\C: \bbS_\C \to \GL(V_\C)\]
of algebraic groups over $\C$ by 
\begin{equation}\label{eq:hodgetype}
   h_\C(z_1,z_2) \cdot v := z_1^{-p} z_2^{-q} v, \quad \forall\, v\in V^{p,q}. 
\end{equation}
The morphism $h_\C$ is defined over $\R$ and descends to a morphism $h:\bbS \to \GL(V)$ of real algebraic groups. 
The category ${\rm Hdg}_\R$ of real Hodge structures is an $\R$-linear tensor category, and the above functor gives an equivalence of $\R$-linear tensor categories between the category ${\rm Hdg}_\R$ and the tensor category ${\rm Rep}_{\R} \bbS$ of real representations of $\bbS$.
We shall also call the representation $(V,h)$ as above a real Hodge structure. We also denote an $R$-Hodge structure by $(L,h)$, where $h:\C^\times \to \GL(L_\R)$ is the corresponding real representation of the real group $\C^\times=\bbS(\R)$. Then a morphism $f:(L_1,h_1)\to (L_2,h_2)$ of $R$-Hodge structures is an $R$-linear map $f:L_1\to L_2$ such that $f_\R \circ h_1(z)=h_2(z)\circ f_\R$ for all $z\in \C^\times$. 

Let $\Z(1)=(2\pi i)\cdot \Z$ denote the Tate twist. It is a $\Z$-rank one $\Z$-Hodge structure of weight~$-2$ with Hodge type $(-1,-1)$. Let $\Z(-1):=\Hom_{\Z}(\Z(1),\Z)=(2\pi i)^{-1}\cdot \Z$ be the dual of $\Z(1)$.
Set $\mathbb Z(0):=\mathbb Z=(2\pi i)^0\mathbb Z$, and, for any $n\in\mathbb Z\setminus\{0\}$, set $\Z(n):=\Z(n/|n|)^{\otimes |n|}=(2\pi i)^n \cdot \Z$ to be the $n$-th power of $\Z(1)$ and set 
$R(n):=\Z(n)\otimes_\Z R=(2\pi i)^n \cdot R$.   
The corresponding character $h:\C^\times \to \GL(\R(n))=\R^\times$ sends $z$ to $(z\bar z)^{n}$.

\def\PHdg{${\rm PHdg}_\Z$\ }
\begin{definition}\label{def:pol}
\begin{enumerate}
    \item A {\it polarisation} on a pure $R$-Hodge structure $(L,h)$ of weight $n$ is an $R$-bilinear pairing
    \[ \psi: L \times L \to R(-n),  \] 
    such that:
    \begin{enumerate}
        \item the induced map 
\[ \psi_\R: L_\R \times_\R L_\R \to \R(-n) \]
        is $\bbS(\R)$-equivariant, that is, 
        \begin{equation}
            \psi_\R(h(z)x,h(z)y)=(z \bar z)^{-n} \psi_{\mathbb{R}}(x, y), \quad \forall\, x,y\in L_\R, \quad z\in \C^\times=\bbS(\R); \text{\ and}  
        \end{equation}
        \item the pairing 
        \begin{equation} \label{eq:varphi}
            \varphi: L_\R \times L_\R \to \R, \quad \varphi(x,y):=(2 \pi i)^n \psi_\R(x,Cy)\in \R, \quad  C=h(i)
        \end{equation}
        is symmetric and positive definite. 
    \end{enumerate}

    \item A {\it polarised $R$-Hodge structure of pure weight $n$}, denoted $(L,\psi, h)$, consists of an $R$-Hodge structure $(L, h)$ of pure weight $n$ together with a polarisation $\psi$ on $(L,h)$. A morphism $f: (L_1,\psi_1,h_1)\to (L_2,\psi_2,h_2)$ of polarised $R$-Hodge structures of pure weight $n$ is a morphism $f:(L_1,h_1) \to (L_2,h_2)$ of $R$-Hodge structures such that 
    \begin{equation}\label{eq:3}
        \psi_2(f(x),f(y))=\psi_1(x,y), \quad  \forall\,x,y\in L_1.
    \end{equation}
    We denote by ${\rm PHdg}_R(n)$ the category of polarised $R$-Hodge structures of pure weight~$n$. 

    \item A {\it polarisation} on an arbitrary $R$-Hodge structure $(L,h)$ is an $R$-bilinear pairing
    \[ \psi: L \times L \to \oplus_{n\in \Z} R(-n),  \] 
    respecting the gradings by weights such that for each $n\in \Z$, its restriction $\psi^n$ to $L^n$ is a polarisation on $(L^n,h|_{L^n})$. That is, the map $\psi^n:L^n\times L^n \to R(-n)$ satisfies Definition~\ref{def:pol}.(1) (a) and (b). Polarised $R$-Hodge structures and morphisms between them are defined in the same way as (2) for pure weights. We denote by ${\rm PHdg}_R$ the category of polarised $R$-Hodge structures. 
    
\end{enumerate}
\end{definition}

Condition (b) implies that $\psi$ is automatically non-degenerate. 
The non-degeneracy of $\psi$ and Condition \eqref{eq:3} imply that every morphism $f$ of polarised $R$-Hodge structures is injective. 
    
A real Hodge structure $h:\C^\times \to \GL(V)$ is said to arise from a complex structure if it can be extended to a morphism $h:\C \to \End(V)$ of $\R$-algebras. This occurs precisely when~$V$ is of Hodge type $\{(-1,0), (0,-1)\}$. By a theorem of Riemann \cite[Theorem 4.7]{deligne:shimura}, there is an equivalence of categories between the category of polarised complex abelian varieties and the category of polarised $\Z$-Hodge structures of Hodge type $\{(-1,0), (0,-1)\}$. In particular, Theorem~\ref{thm:1.5} generalises Theorem~\ref{thm:1.3} in the case where $K=\C$.

\subsection{Uniqueness of orthogonal decompositions}
\label{sec:P.2}\

From now on we consider $R$-Hodge structures with $R=\Z$ and call them integral Hodge structures. 

\begin{definition}

\begin{enumerate}
    \item An \emph{orthogonal decomposition} of an  object $(L,\psi,h)$ in ${\rm PHdg}_\Z$ is $L=\oplus_{\nu=1}^r L_\nu$, where each $L_\nu$ is a $\Z$-Hodge substructure (that is, a $\Z$-sublattice in $L$ such that 
    $L_{\nu,\R}$ is stable under the action of $h$), such that $\psi(L_\nu,L_\mu)=0$ for $\nu\neq \mu$.
    We also write $L=\bot_{\nu=1}^r L_\nu$. 

    \item A nonzero object $(L,\psi,h)$ in ${\rm PHdg}_\Z$ is said to be \emph{indecomposable} if whenever
   $L=L_1\bot L_2$ is an orthogonal decomposition of $L$ in ${\rm PHdg}_\Z$, one has either $L_1=0$ or $L_2=0$. 
\end{enumerate}   
\end{definition}
 




\begin{theorem}\label{thm:5.4}
Let $(L,\psi,h)$ be a polarised pure integral Hodge structure, 
and let 
\begin{equation}
L=\bot_{\nu=1}^r L_\nu     
\end{equation}
be an orthogonal decomposition into indecomposable polarised  integral Hodge substructures. Then the set of sublattices $L_\nu$ is uniquely determined by $L$.
\end{theorem}
\begin{proof}
    Let $R:=\End(L, h)$  
     be the endomorphism ring of $(L,h)$ and $R^0:=\End(L_\Q, h)$ be the endomorphism algebra of $(L_\Q,h)$.
     Since $(L_\Q, h)$ is polarised, 
     $R^0$ is a semi-simple finite-dimensional $\Q$-algebra. We denote the Rosati involution (the adjoint) on $R^0$ with respect to $\psi$ by $*$; i.e.,  $\psi_{\mathbb Q}(ax,y)=\psi_{\mathbb Q}(x,a^*y)$. Since elements of $R^0$ commute with $C = h(i)$, the involution $*$ is also the adjoint with respect to the positive-definite form $\varphi$ and hence it is a positive involution. 
    
    As in the proof of Theorem~\ref{thm:1.3} by Theorem~\ref{thm:UniDecIdem}, we shall show that every orthogonal decomposition $1=i_1+i_2+\dots+i_r$ with nonzero $i_\nu\in I^*(R)$ gives rise to a nontrivial orthogonal decomposition of polarised $\Z$-Hodge structures $L=\bot_{\nu=1}^r L_\nu$ and vice versa.
    Let $L=\bot_{\nu=1}^r L_\nu$ be an orthogonal decomposition in ${\rm PHdg}_\Z$. 
    For each $1\le \nu\le r$, set 
    \[ i_\nu:L\xrightarrow{p_{L_\nu}} L_\nu \xrightarrow{i_{L_\nu}} L, \] 
    where $p_{L_\nu}$ is the orthogonal projection of $L$ onto $L_\nu$ and $i_{L_\nu}:L_\nu\hookrightarrow L$ is the inclusion. Then $1=i_1+i_2+\dots+i_r$ and $i_\nu^2=i_\nu$ for all $\nu$. For any $x_\mu\in L_\mu$ and $y\in L$, one checks that 
    \[ \psi(y,i_\nu^*(x_\mu))=\begin{cases}
        0 & \text{if $\mu\neq \nu$}; \\
        \psi(y, x_\mu) & \text{if $\mu= \nu$}.
    \end{cases}
    \]
    Then $i_\nu^*=i_\nu$ and $i_\nu\bot i_\mu$ for $\nu\neq \mu$. Thus, we have an orthogonal decomposition $1=i_1+i_2+\dots+i_r$ with $i_\nu\in I^*(R)$.    

    Conversely, if an orthogonal decomposition $1=i_1+i_2+\dots+i_r$ with $i_\nu\in I^*(R)$ is given, we set $L_\nu:=i_\nu L$. Then $L_\nu\bot L_\mu$ for $\nu\neq \mu$ and $L=\sum L_\nu$. Since \[ h(z)L_{\nu,\R}=h(z) i_\nu L_\R= i_\nu h(z) L_\R=i_\nu L_\R=L_{\nu, \R}, \quad z\in \C^\times, \]
    we see that $L=\bot_{\nu} L_\nu$ is an orthogonal decomposition in ${\rm PHdg}_\Z$. 
    Moreover, one checks that these two constructions are inverses of each other. This shows the desired one-to-one correspondence and proves the theorem. 
\end{proof}

\begin{proof}[Proof of Theorem~\ref{thm:1.5}]
We have the decomposition of 
$\Z$-Hodge structures $L=\oplus_{n\in \Z} L^n$ by weights. For different weights $n$ and $m$, we have $\psi(L^n,L^m)=0$. This reduces the problem to the case of pure weight, which is Theorem~\ref{thm:5.4}. 
\end{proof}

\providecommand{\bysame}{\leavevmode\hbox to3em{\hrulefill}\thinspace}
\providecommand{\MR}{\relax\ifhmode\unskip\space\fi MR }
\providecommand{\MRhref}[2]{%
  \href{http://www.ams.org/mathscinet-getitem?mr=#1}{#2}
}
\providecommand{\href}[2]{#2}

\end{document}